\documentclass[11pt,a4paper,twoside]{article}
\usepackage{bm, amsmath, amssymb, amsthm} %, amsthm

 \newtheorem{theorem}{Theorem}
 \newtheorem{corollary} {Corollary}
 \newtheorem{lemma} {Lemma}
 \newtheorem{proposition}{Proposition}
 \newtheorem{definition}{Definition}
 \newtheorem{remark}{Remark}
 \newtheorem{example}{Example}
\def\calf{{\mathcal F}}
\def\vol{\operatorname{vol}}
\newcommand\id{\operatorname{id}}
\def\RR{\mathbb{R}}
\def\<{\langle}
\def\>{\rangle}
\newcommand\tr{\operatorname{tr}}
\newcommand\Div{\operatorname{div}}
\def\Ric{\operatorname{Ric}}
\newcommand{\eps}{\varepsilon}
\newcommand{\eq}{\hspace*{-.4mm}&=&\hspace*{-.4mm}}

\author{Vladimir Rovenski\footnote{Department of Mathematics, University of Haifa, Mount Carmel, 3498838 Haifa, Israel
       \newline e-mail: {\tt vrovenski@univ.haifa.ac.il} } }

\title{A series of integral formulas for a foliated sub-Riemannian manifold}

\begin{document}

\date{}

\maketitle

\begin{abstract}

In this article, we prove a series of integral formulae for  a codimension-one foliated sub-Riemannian manifold, which is a new geometric object
denoting a Riemannian manifold $(M,g)$ equipped with a distribution ${\mathcal D}=T\calf\oplus\,{\rm span}(N)$, where $\calf$ is a foliation of $M$ and $N$ a unit vector field $g$-orthogonal to $\calf$.
Our integral formulas involve $r$th mean curvatures of $\calf$, Newton transformations of the shape operator of $\calf$ with respect to $N$ and the curvature tensor of induced connection on ${\mathcal D}$ and generalize some known integral formulas (due to Brito-Langevin-Rosenberg, Andrzejewski-Walczak and the author)
for codimension-one foliations.
We apply our formulas to sub-Riemannian manifolds with restrictions on the curvature and extrinsic geometry of a foliation.


%; 57R25

\end{abstract}

%%% ----------------------------------------------------------------------
\maketitle
%%% ----------------------------------------------------------------------
%\tableofcontents

\section*{Introduction}

Integral formulas are useful for solving many problems in Riemannian geometry, both manifolds (for example, the Gauss-Bonnet formula for closed surfaces) and foliations, see surveys \cite{arw,RWa-1}:
%prevents the specification of the Gaussian curvature of

\smallskip\noindent\
$\bullet$~characterizing of foliations, whose leaves have a given geometric property;

\noindent\
$\bullet$~prescribing the higher mean curvatures of the leaves of a foliation;

\noindent\
$\bullet$~minimizing functionals like volume defined for tensor fields on a foliation.

\smallskip

The~first known integral formula for a codimension-one foliation of a closed Riemannian manifold belongs to G.\,Reeb, see \cite{reeb1},
\begin{equation}\label{E-sigma1}
 \int_M H\,d\vol_g=0;
\end{equation}
thus,
%the total mean curvature $H$ of the leaves is equal to zero,
either the mean curvature of the leaves $H\equiv0$ or $H(x)\,H(x')<0$ for some points $x\ne x'$ on $M$.
Recall that any compact manifold with Euler number equal to zero admits codimension-one foliations.
 There is a series of integral formulas for codimension-one foliations,
starting with \eqref{E-sigma1}, with consequences for foliated space forms, see surveys in \cite{arw,blr,RWa-1}.
These integral formulas were obtained by applying the Divergence Theorem to suitable vector fields.
The~second formula in the series of total $r$th mean curvatures $\sigma_r$'s
(symmetric functions of the principal curvatures of the leaves) of a codimension-one foliation with a unit normal vector field $N$~is,
see~\cite{nora},
\begin{equation}\label{E-sigma2}
 \int_M (2\,\sigma_2-\Ric_{N,N})\,{\rm d}\vol_g=0.
\end{equation}
By \eqref{E-sigma2}, there are no totally umbilical codimension-one foliations of a closed manifold of negative Ricci curvature,
and there are no harmonic codimension-one foliations of a closed manifold of positive Ricci curvature.

In \cite{aw2010}, the Newton transformations $T_r(A)$ (of the \textit{Weingarten operator} $A$ of the leaves) were applied to codimension one foliations,
and a series of integral formulas for $r\ge0$ starting with \eqref{E-sigma2} was obtained (with consequences for foliated space forms,
see~\cite{blr}):
\begin{equation}\label{E-intNTp1}
 \int_{M}\big((r+2)\sigma_{r+2}-\tr_{\,\calf}(T_r(A)\,{\mathcal R}_{N})
 -\sum\nolimits_{1\le j\le r}(-1)^{j-1}\tr_{\,\calf}(T_{r-j}(A)\,{\mathcal R}_{A^{j-1}\nabla_{N} N}\big)\,{\rm d}\vol_g=0.
\end{equation}
Here, the linear operator ${\mathcal R}_{X}:T\calf\to T\calf$ for $X\in TM$ is given by
%\begin{equation*}
%\label{E-operRYN}
 ${\mathcal R}_{\,X}:Y\to R(Y, X)N$.

 The natural question
 %on foliated Riemannian manifolds
 arises:
%\smallskip
{can one find similar integral formulas for foliations of arbitrary codimension}?
%\smallskip\noindent
In studying this question, we consider a foliation $\calf$ of $(M,g)$ equipped with a unit normal vector field $N$. In this case, $T\calf$
(the tangent bundle of $\calf$) and $N$ span in the general case a non-integrable distribution ${\mathcal D}$ on $(M,g)$ (subbundle of $TM$).
Note that $(M,{\mathcal D})$ with the metric on ${\mathcal D}$ (e.g., the restriction of $g$) is the main object of sub-Riemannian geometry, see~\cite{BF}.
Thus, we study the following sub-Riemannian version of the above question:
%\smallskip
\textit{can one extend integral formulas \eqref{E-intNTp1} for codimension-one foliated sub-Riemannian manifolds}?
%\smallskip\noindent

Apparently, a foliated sub-Riemannian manifold $(M,{\mathcal D},\calf,g)$, i.e., the tangent bundle $T\calf$ of a foliation $\calf$ is a subbundle of ${\mathcal D}$, is a new geometrical object.
In~the article, we consider a codimension-one foliated sub-Riemannian manifold $(M,{\mathcal D},\calf,g)$,
i.e., ${\mathcal D}=T\calf\oplus\,{\rm span}(N)$ and $N$ is a unit vector field $g$-orthogonal to $\calf$,
 and prove a series of integral formulas (in Theorems~\ref{T-mainLW-1}--\ref{T-main01p1} and Corollaries~\ref{C-P1},~\ref{C-P2}),
which generalize known integral formulas, see \cite{aw2010,blr,reeb1,RWa-1}, for codimension-one foliated Riemannian manifolds.
The~integral formulas involve $r$th mean curvatures of $\calf$ (i.e., symmetric functions of the principal curvatures of the leaves),
Newton transformations of the shape operator of $\calf$ with respect to $N$ and the curvature tensor of induced connection on ${\mathcal D}$.
We apply our formulas to sub-Riemannian manifolds
with restrictions on the curvature and extrinsic geometry of a foliation.
%of constant modified curvature and totally umbilical foliations.

%Since $(M, g, {\mathcal D})$, where ${\mathcal D}=T\calf\oplus\,{\rm span}(N)$, is a sub-Riemannian manifold, our results can be used in sub-Riemannian geometry.

%The Reeb type integral formula
\section{Preliminaries}
%\label{sec:02}

Here, we use the induced linear connection to define the shape operator (with its Newton transformations) and the curvature tensor related to a codimension-one foliated sub-Riemannian manifold, then we prove three auxiliary lemmas.

Let ${\mathcal D}$ be an $(n+1)$-dimensional distribution on a smooth $m$-dimensional manifold $M$, i.e.,
a subbundle of $TM$ of rank $n+1$ (where $0<n<m$).
In other words, to each point $x\in M$ we assign an $(n+1)$-dimensional subspace ${\mathcal D}_x$ of the tangent space $T_xM$ smoothly depending on $x$.
%%%%%%%%%%
An integrable distribution determines a foliation; in this case the Lie bracket of any two vector fields from ${\mathcal D}$ also belongs to ${\mathcal D}$.
A pair $(M,{\mathcal D})$, where $M$ is a manifold and ${\mathcal D}$ is a non-integrable distribution on $M$, is called a \textit{non-holonomic manifold}, see~\cite{BF}.
The notion of non-holonomic manifold was introduced for a geometric interpretation of constrained systems in classical mechanics.
%%%%%%%%%%
A~\textit{sub-Riemannian manifold} is a non-holonomic manifold $(M,{\mathcal D})$, equipped with a sub-Riemannian metric $g=\<\cdot\,,\cdot\>$, i.e.,
the scalar product $g:{\mathcal D}_x\times{\mathcal D}_x\to\RR$ for all $x\in M$, see \cite{BF}.
Usually, they assume that the sub-Riemannian metric on ${\mathcal D}$ (the horizontal bundle)
is extended to a Riemannian metric on the whole space $M$, also denoted by $g$.
%This allows us to define the orthogonal distribution ${\mathcal D}^\bot$ (the vertical subbundle) such that $TM={\mathcal D}\oplus{\mathcal D}^\bot$.

The orthoprojector $P:TM\to{\mathcal D}$  onto the distribution ${\mathcal D}$ is characterized by the properties
\begin{equation*}
%\label{E-PP}
 P=P^*\ (\textrm{self-adjoint}),\quad P^2=P,
\end{equation*}
e.g., \cite{g1967},
and similarly for the (1,1)-tensor $P^\bot=\id_{\,TM}-{P}$ -- the orthoprojector onto the orthogo\-nal distribution ${\mathcal D}^\bot$ (the vertical subbundle).
A sub-Riemannian manifold $(M,{\mathcal D},g)$ equipped with a foliation $\calf$ such that the tangent bundle $T\calf$ is a subbundle of ${\mathcal D}$ will be called a \textit{foliated sub-Riemannian manifold}.

In this article, we assume that $\dim\calf=n$ (that is $\calf$ is a codimension-one foliation relative to ${\mathcal D}$),
and there exists a unit vector field $N$ orthogonal to $\calf$ and tangent to ${\mathcal D}$. This means that the Euler characteristic of ${\mathcal D}$ is zero.
%Let $\calf$ be a foliation of $(M,g)$, and let $N$ be .
%Put
%\[
% {\mathcal D}_1 = T\calf,\quad {\mathcal D}_2 = {\rm span}(N),
%\]
%thus,
The following orthogonal decomposition is valid:
\[
 {\mathcal D}=T\calf\oplus\,{\rm span}(N).
\]
The Levi-Civita connection $\nabla$ on $(M,g)$ induces a linear connection $\nabla^P$ on ${\mathcal D}=P(TM)$:
\[
 \nabla^P_X PY = P\nabla_X PY,\quad X,Y\in\Gamma(TM),
\]
which is compatible with the metric: $X\<U,V\> = \<\nabla^P_X U,\, V\> +\<U,\, \nabla^P_X V\>$.

Define the horizontal vector field
\[
% Z=P\,\nabla_N\,N,
 Z=\nabla^P_N\,N,
\]
and note that $\nabla_N\,N$ is the curvature vector of $N$-curves in $(M,g)$.

The \textit{shape operator} $A_N: T\calf\to T\calf$ of the foliation $\calf$ with respect to $N$ is defined by
\begin{equation}\label{E-A-D}
 A_N(X)=-\nabla^P_X\,N,\quad X\in T\calf.
% A_N(X)=-P(\nabla_X\,N),\quad X\in T\calf.
\end{equation}
The \textit{elementary symmetric functions} $\sigma_j(A_N)$ of $A_N$
($r$th mean curvatures of $\calf$ in ${\mathcal D}$)
are given~by the equality
\[
 \sum\nolimits_{\,r=0}^{\,n}\sigma_r(A_N)\,t^r=\det(\,\id_{\,T\calf}+\,t\,A_N),\quad t\in\RR.
\]
Note that $\sigma_r(A_N)=\sum\nolimits_{\,i_1<\cdots<i_r} \lambda_{i_1}\cdots \lambda_{i_r}$,
%\]
where $\lambda_1\le\ldots\le\lambda_n$ are the eigenvalues of $A_N$.
Let $\tau_j(A_N)=\tr\,(A_N)^j$ for $j\in\mathbb{N}$ be the \textit{power sums} symmetric functions of $A_N$.
For~example, $\sigma_0(A_N)=1$, $\sigma_1(A_N)=\tau_1(A_N)=\tr A_N$, $\sigma_n(A_N)=\det A_N$, and
\begin{equation}\label{E-sigma2-tau1-2}
 2\,\sigma_2(A_N)=\tau_1^2(A_N)-\tau_2(A_N).
\end{equation}
For short, we set
\[
 A=A_N,\quad
 \sigma_{r}=\sigma_{r}(A_N),\quad
 \tau_{r}=\tau_{r}(A_N).
\]
Next, we introduce the curvature tensor $R^P:TM\times TM\rightarrow{\rm End}({\mathcal D})$ of the connection $\nabla^P$:
\begin{equation}\label{E-R-k}
%\label{E-R-ki}
 {R}^P(X,Y) = \nabla^P_{X}\nabla^P_{Y} -\nabla^P_{Y}\nabla^P_{X} -\nabla^P_{[X,\, Y]};
\end{equation}
Set ${R}^P(X,Y,V,U)=\<{R}^P(X,Y)V,\, U\>$ for $U,V\in{\mathcal D}$.
Obviously, ${R}^P(Y,X)V=-{R}^P(X,Y)V$ (for any linear connection). Since $\nabla^P$ is compatible with the metric, then
\begin{equation}\label{E-R-symm2}
 {R}^P(X,Y,V,U)=-{R}^P(X,Y,U,V).
\end{equation}
Recall the Codazzi's equation for a codimension-one foliation (or a submanifold) of $(M,g)$:
\begin{equation}\label{E-codazzi-classic}
 (\nabla_{X}\, h)(Y,U) -(\nabla_{Y}\, h)(X,U) = (R(X,Y)Z)^\bot,
\end{equation}
where
$R:TM\times TM\rightarrow{\rm End}(TM)$ is the Riemann curvature tensor of the Levi-Civita connection,
\[
 {R}(X,Y) = \nabla_{X}\nabla_{Y} -\nabla_{Y}\nabla_{X} -\nabla_{[X,\, Y]},
\]
$^\bot$ denotes the projection onto the vector bundle orthogonal to $\calf$,
and $h:T\calf\times T\calf\to(T\calf)^\bot$ is the second fundamental form of $\calf$ in $(M,g)$ defined by
\[
 h(X,Y)=(\nabla_XY)^\bot.
\]

\begin{lemma}
%  (\ref{E-codazziCN}),
For ${\mathcal D}=T\calf\oplus\,{\rm span}(N)$ on $(M,g)$, the following Codazzi type equation is valid:
\begin{equation}\label{E-codazziAN}
 (\nabla^\calf_{X}\, A) Y -(\nabla^\calf_{Y}\, A) X = -R^P(X,Y)N ,\quad X,Y\in T\calf.
\end{equation}
\end{lemma}

\begin{proof}
From \eqref{E-codazzi-classic}, for all vectors $X,Y,U\in T\calf$ we get
\begin{equation}\label{E-codazzi-1}
 \<(\nabla^\calf_{X}\, A) Y -(\nabla^\calf_{Y}\, A) X,\ U\> + \<h(X,U), \nabla_Y N\> - \<h(Y,U), \nabla_X N\>
 = -\<R(X,Y)N, U\> .
\end{equation}
 Applying the orthoprojector
on the vector bundle orthogonal to $\calf$,
we find
\begin{equation}\label{E-codazzi-RN}
 \<R(X,Y)N, U\>=\<R^P(X,Y)N, U\> +
 \<\nabla_{X}((\nabla_{Y}\,N)^\bot) -\nabla_{Y}((\nabla_{X}\,N)^\bot) -\nabla_{[X,\ Y]^\bot}\,N, U\>.
\end{equation}
Using the equalities $[X,\ Y]^\bot=0$ (since $T\calf$ is integrable), \eqref{E-codazzi-RN} and
\begin{eqnarray*}
 \<h(X,U),\, \nabla_Y\,N\> = \<\nabla_{Y}N,\, (\nabla_{X}\,U)^\bot\> = -\<\nabla_{X}((\nabla_{Y}\,N)^\bot),\, U\>,\\
 \<h(Y,U),\, \nabla_X\,N\> = \<\nabla_{X}N,\, (\nabla_{Y}\,U)^\bot\> = -\<\nabla_{Y}((\nabla_{X}\,N)^\bot),\, U\>,
\end{eqnarray*}
in \eqref{E-codazzi-1} completes the proof.
\end{proof}

\smallskip

The following lemma generalizes \cite[Lemma~3.1]{lw2}.

\begin{lemma}\label{L-31-myC}
Let $\{e_i\}$ be a local orthonormal frame of $\,T\calf$ such that  at a point $x\in M$:

$\bullet$ $\nabla^\calf_X\,e_i=0\ (1\le i\le n)$ for any vector $X\in T_xM$;

$\bullet$ $\nabla^P_\xi\,e_i=0\ (1\le i\le n)$ for any vector $\xi\in {\mathcal D}^\bot_xM$.

\noindent
Then the following equality is valid at $x\in M$:
\begin{equation}\label{E-EiNNC}
 \<\nabla_{e_i} Z, e_j\> = \<A^2 e_{i}, e_{j}\> + \<{R}^P(e_i,N)N, e_j\> - \<(\nabla^{\calf}_N \,A)e_i,e_j\> +\<Z, e_i\>\<Z, e_j\>.
\end{equation}
\end{lemma}

\proof Taking covariant derivative of $\<Z,\, e_j\>=-\<N,\, \nabla_{N}\,e_j\>$ with respect to $e_i$, we find
\begin{equation}\label{E-aw-08C}
 -\<Z, \nabla_{e_i}e_j\>= \<\nabla_{e_i} Z, e_j\> +\<\nabla_{e_i}N, P\nabla_{N}\,e_j\> +\<N, \nabla_{e_i}P\nabla_{N}\,e_j\>.
\end{equation}
For a foliation $\calf$, we obtain
\[
% (\nabla^{\calf}_N A)_{ij}
 \<(\nabla^{\calf}_N A)e_i,e_j\> =\nabla_{N}\<N, \nabla_{e_i}\,e_j\>=\<Z, \nabla_{e_i}e_j\> +\<N,\nabla_{N}P\nabla_{e_i}\,e_j\>.
\]
Therefore, using \eqref{E-R-k}, we calculate at the point $x\in M$:
\begin{eqnarray}\label{E-aw-09C}
\nonumber
 && \<A^2 e_{i}, e_{j}\> +\<{R}^P(e_i,N)\,N, e_j\> -\<(\nabla^{\calf}_N A)e_i,e_j\> \\
\nonumber
 && =\<A^2 e_{i}, e_{j}\>-\<{R}^P(e_i,N)\,e_j, N\>+N\<\nabla_{e_i} N, \,e_j\>\\
 && =\<A^2 e_{i}, e_{j}\>-\<Z, \nabla_{e_i}e_j\>-\<\nabla_{e_i}P\nabla_{N}\,e_j, N\> +\<\nabla_{[e_i,N]}\,e_j, N\>.
\end{eqnarray}
Since $\<\nabla_{P^\bot[e_i,N]}\,e_j, N\>=0$, see conditions at $x\in M$, we rewrite the last term in \eqref{E-aw-09C} as
\[
 \<\nabla_{P[e_i,N]}\,e_j, N\> =\<\nabla_{[e_i,N]}\,e_j, N\>.
\]
Then, using (\ref{E-aw-08C}) and the following equalities at $x\in M$:
\begin{eqnarray*}
 && P\nabla_{e_i}N=\sum\nolimits_{\,1\le j\le n}\<\nabla_{e_i}N, e_j\>e_j,\quad
 P\nabla_{N}\,e_i=\<\nabla_{N}\,e_i, N\>N,\\
 && \<A^2 e_{i}, e_{j}\>=\<\nabla_{e_i}N, N\>\<\nabla_{N}\,e_j, N\>,
\end{eqnarray*}
we simplify the last line in (\ref{E-aw-09C}) as
\begin{equation*}
 \<\nabla_{e_i}Z, e_j\>-\<Z, e_i\>\<Z, e_j\>.
\end{equation*}
From the above, the claim follows.
\hfill$\square$

\smallskip

Many authors investigated $r$th mean curvatures of foliations and hypersurfaces of Riemannian manifolds using the Newton transformations of the shape operator, see \cite{aw2010}.

\begin{definition}\rm
The \textit{Newton transformations} $T_{r}(A)$ of the shape operator $A$ of an $n$-dimensional foliation $\calf$
of a sub-Riemannian manifold $(M,{\mathcal D},g)$ are defined recursively or explicitly by
\begin{equation*}
%\label{E-defNTind}
  T_0(A) = \id_{\,T\calf},\quad T_r(A)=\sigma_r\id_{\,T\calf} -A\,T_{r-1}(A),\quad 1\le r\le n,
\end{equation*}
\begin{equation*}
%\label{E-intNT-C}
 T_r(A) = \sum\nolimits_{j=0}^r(-1)^j\sigma_{r-j}\,A^j =\sigma_r\id_{\,T\calf} -\sigma_{r-1}\,A+\ldots+(-1)^r A^{\,r}.
\end{equation*}
\end{definition}

For example, $T_1(A)=\sigma_1\id_{\,T\calf} - A$ and $T_n(A)=0$. Notice that $A$ and $T_{r}(A)$ commute.

\begin{lemma}[see Lemma~1.3 in \cite{RWa-1}]\label{L-NTprop}
For the shape operator $A$ we have
 \begin{eqnarray*}
 \tr_{\,\calf} T_r(A)\eq(n-r)\,\sigma_r,\\
 \tr_{\,\calf} (A\cdot T_r(A)) \eq (r+1)\,\sigma_{r+1},\\
 \tr_{\,\calf} (A^2\cdot T_r(A)) \eq \sigma_{1}\,\sigma_{r+1}-(r+2)\,\sigma_{r+2},\\
 \tr_{\,\calf} ( T_{r-1}(A)(\nabla_X^\calf\,A)) \eq X(\sigma_{r}),\quad X\in T\calf.
 \end{eqnarray*}
\end{lemma}

On the other hand, since the (1,1)-tensors $A$ and $T_r(A)$ are self-adjoint, we have
\begin{equation}\label{E-Tr-adjoint}
 \<(\nabla^\calf_{X}\,T_r(A))Y, \,V\> = \<(\nabla^\calf_{X}\,T_r(A))V, \,Y\>,\quad X,Y,V\in T\calf.
\end{equation}

\section{Main results}
%\label{sec:03}

Here, we prove a series of integral formulas for a codimension-one foliated sub-Riemannian mani\-fold $(M,{\mathcal D},\calf,g)$
with ${\mathcal D}=T\calf\oplus\,{\rm span}(N)$.
%On the other hand, ${\mathcal F}$ is a foliation of a Riemannian manifold $(M,g)$, and $N$ is a unit vector field orthogonal to ${\mathcal F}$
%such that ${\mathcal D}=T\calf\oplus\,{\rm span}(N)$ is a distribution on $M$.

Recall that the $\calf$-{divergence} of a vector~field $X$ on $(M,g)$ is defined by
\[
 \Div\,X
 %=\operatorname{trace}(Y\to\nabla_{Y} X)
 =\sum\nolimits_{\,1\le i\le n}\<\nabla_{e_i}\,X, \,e_i\>,
\]
where $\{e_i\}$ is a local orthonormal frame of $T\calf$.
Following \cite{lw2}, define the $\calf$-divergence of the Newton transformation $T_r(A)$ by
\begin{equation*}
%\label{eq:div1}
 \Div_\calf T_r(A) = \sum\nolimits_{\,1\le i\le n}(\nabla^\calf_{e_i}\,T_r(A))\,e_i.
\end{equation*}
For any $X\in{\mathcal D}$, define a linear operator ${\mathcal R}^P_{X}:T\calf\to T\calf$ by
\begin{equation*}
%\label{E-operRYN}
 {\mathcal R}^P_{\,X}: V \to R^P(V, X)N,\quad V\in T\calf.
\end{equation*}
%Note that if $\calf$ is a codimension-one foliation of $M$, then ${\mathcal R}^P_{\,X}=\Ric(X,N)$.
We can view $\tr_{\,\calf}{\mathcal R}^P_{X}=\sum\nolimits_{\,1\le i\le n}\<R^P(e_i,X)N, \,e_i\>$
as the \textit{Ricci $P$-curvature} $\Ric^P_{X,N}$.
% in the $N$-direction.
Then $\Ric^P_{N,N}=\tr_{\,\calf}{\mathcal R}^P_{N}$
%=\sum\nolimits_{\,1\le i\le n}\<R^P(e_i,N)N, \,e_i\>$
is the Ricci $P$-curvature in the $N$-direction.

The following result generalizes \cite[Lemma~2.2]{lw2}.

\begin{proposition}
%\label{L-divNT-1}
The leafwise divergence of $T_r(A)$ satisfies the inductive formula
\begin{eqnarray}\label{E-divNT-proof2}
 \<\Div_\calf T_r(A), X\> \eq -\<\Div_\calf T_{r-1}(A), \,A X\> +\tr_{\,\calf}(T_{r-1}(A)\,{\mathcal R}^P_{\,X}),
 %\sum\nolimits_{\,1\le i\le n}\<R^P(e_i, X)N, \,T_{r-1}(A) e_i\>,
\end{eqnarray}
for $r>0$ and any vector field $X\in\Gamma(\calf)$; moreover, $\Div_\calf T_0(A)=0$. Equivalently,
\begin{equation}\label{E-divNTN-Z}
 \<\Div_\calf T_r(A),\,X\> = \sum\nolimits_{\,1\le j\le r} (-1)^{j-1}\tr_{\,\calf}(T_{r-j}(A)\,{\mathcal R}^P_{\,A^{j-1} X}).
\end{equation}
\end{proposition}

\begin{proof}
 Using the inductive definition of $T_r(A)$, we have
\begin{equation*}
 \Div_\calf T_r(A)=\nabla^\calf\sigma_r-A\Div_\calf T_{r-1}(A)
 -\sum\nolimits_{\,1\le i\le n}(\nabla^\calf_{e_i}\,A)\,T_{r-1}(A) e_i.
\end{equation*}
 %Similarly to the proof of Lemma~\ref{L-divCk-1},
Using Codazzi type equation \eqref{E-codazziAN} and the last formula in Lemma~\ref{L-NTprop}, we obtain
\begin{eqnarray*}
 &&\sum\nolimits_{\,1\le i\le n}\<(\nabla^\calf_{e_i}\,A) T_{r-1}(A) e_i, \,X\>
 =\sum\nolimits_{\,1\le i\le n}\<T_{r-1}(A) e_i, \,(\nabla^\calf_{e_i}\,A)X\>\\
 && =\sum\nolimits_{\,1\le i\le n} \<T_{r-1}(A)e_i,\,(\nabla^\calf_{X}A) e_i -R^P(e_i, X)N\> \\
 &&=\tr_{\,\calf}(T_{r-1}(A)(\nabla^\calf_{X}\,A)) -\sum\nolimits_{\,1\le i\le n}\<R^P(e_i, X)N,\, T_{r-1}(A) e_i\>\\
 && = X(\sigma_r)-\tr_{\,\calf}(T_{r-1}(A)\,{\mathcal R}^P_{\,X}).
\end{eqnarray*}
%By Lemma~\ref{L-NTprop}, we have $X(\sigma_r)=\tr_{\,\calf}(T_{r-1}(A)\nabla^\calf_{X} A)$ for any $X\in T\calf$.
Hence, the inductive formula (\ref{E-divNT-proof2}) holds. Finally, \eqref{E-divNTN-Z} follows directly from the above.
\end{proof}

\begin{remark}\rm
(a) If $P=\id_{\,TM}$, i.e., $\calf$ is a codimension-one foliation of $M$, then $R^P=R$, and using the symmetry
$\<R(X,Y)U, \,V\>=\<R(U,V)X, \,Y\>$, we simplify equation \eqref{E-divNT-proof2} to the form
\begin{equation*}
%\label{E-divNTk-1}
 \Div_\calf T_r(A) = - A\Div_\calf T_{r-1}(A) +\sum\nolimits_{\,1\le i\le n} (R(N, \,T_{r-1}(A) e_i)e_i)^\top,
\end{equation*}
where $^\top$ denotes the orthogonal projection on the vector bundle $T\calf$, see \cite[Lemma~2.2]{lw2}.

(b) Let the distribution $T\calf$ be $P$-\textit{curvature invariant}, that is
\begin{equation}\label{E-Pcurv-inv}
 R^P(X,Y)V \in T\calf,\quad X,Y,V\in T\calf.
\end{equation}
In view of \eqref{E-R-symm2}, equation \eqref{E-divNTN-Z} implies that $\Div_\calf T_r(A)=0$ for every $r\ge0$.
Condition \eqref{E-Pcurv-inv} is obviously satisfied, if the distribution $T\calf$ is
auto-parallel, i.e., $\nabla_XY\in\Gamma(T\calf)$ for all $X,Y\in\Gamma(T\calf)$.
A~sufficient condition for \eqref{E-Pcurv-inv} is the following equality (for some real constant $c$):
\begin{equation}\label{E-Pcurv-c}
 R^P(X,Y)V= c\,(\/\<Y,V\>X-\<X,V\>Y\/),\quad X,Y,V\in {\mathcal D}.
\end{equation}
%, in particular, $(M,{\mathcal D},g)$ has constant sectional $P$-curvature.
\end{remark}

The following result generalizes \cite[Proposition~3.3]{lw2}.

\begin{proposition}\label{P-Cintfiv01p1}
We have
\begin{eqnarray*}
 \Div_{\calf}(T_r(A) Z) \eq\underline{\<\Div_\calf T_r(A), \,Z\>}
 +\tr_{\,\calf}(T_r(A) {\mathcal R}^P_{N}) + \<T_r(A) Z, Z\>  \\
 && -\,(r+2)\sigma_{r+2}-N(\sigma_{r+1})+\sigma_{1}\sigma_{r+1},
\end{eqnarray*}
where $Z=P\,\nabla_NN$ and the underlined term is given by \eqref{E-divNTN-Z} with $X=Z$.
\end{proposition}

\proof Using \eqref{E-Tr-adjoint}, we can compute the divergence of the vector field $T_r(A)Z$ as follows:
\begin{equation*}
 \Div_\calf T_r(A) Z = \sum\nolimits_{\,1\le i\le n}\<\nabla_{e_i} (T_r(A) Z), e_i\>
 = \<\Div_\calf T_r(A),\, Z\> +\sum\nolimits_{\,1\le i\le n} \<\nabla_{e_i}{Z},\, T_r(A) e_i\>.
\end{equation*}
Using (\ref{E-EiNNC}) of Lemma \ref{L-31-myC}, we compute $\sum\nolimits_{\,1\le i\le n}\<\nabla_{e_i}{Z}, \,T_r(A) e_i\>$ as
\begin{eqnarray*}
 &&\sum\nolimits_{\,1\le i\le n}\big(\<A^2 e_i + {R}^P(e_i,N)N -(\nabla^\calf_N A) e_i, \,T_r(A) e_i\> +\<Z, e_i\>\,\<Z, \,T_r(A) e_i\>\big)\\
 &&\qquad= -\tr_{\,\calf}\big(T_r(A)(\nabla^\calf_N A-A^2-{\mathcal R}^P_{N})\big) +\<T_r(A)Z,\, Z\>.
\end{eqnarray*}
By Lemma~\ref{L-NTprop}, we can write
\begin{equation*}
 \tr_{\,\calf}\big( T_r(A)(\nabla_N A-A^2-{\mathcal R}^P_{N})\big) = N(\sigma_{r+1})-\sigma_{1}\sigma_{r+1}+(r+2)\sigma_{r+2}
 -\tr_{\,\calf}(T_r(A)\,{\mathcal R}^P_{N}).
\end{equation*}
 Finally, we have
\begin{eqnarray*}
 && \sum\nolimits_{\,1\le i\le n}\<\nabla_{e_i}{Z},\, T_r(A) e_i\> = -N(\sigma_{r+1})+\sigma_{1}\,\sigma_{r+1}-(r+2)\sigma_{r+2}\\
 && +\,\<T_r(A)Z, \,Z\>
 +\sum\nolimits_{\,1\le i\le r}(-1)^{j-1}\tr_{\,\calf}(T_{r-j}(A)\,{\mathcal R}^P_{A^{j-1}Z}).\quad\qed
\end{eqnarray*}

From the above, applying the Divergence Theorem to any compact leaf, we get

\begin{theorem}\label{T-mainLW-1}
For any compact leaf $L$ of $\calf$ and $Z=P\,\nabla_NN$ we have
\begin{eqnarray*}
 &&\int_{L}\big(
 (r+2)\,\sigma_{r+2} +N(\sigma_{r+1}) -\sigma_{1} \sigma_{r+1} -\tr_{\,\calf}(T_r(A){\mathcal R}^P_{N}) \\
 &&\  -\,\<T_r(A) Z, Z\> -\sum\nolimits_{\,1\le i\le r}(-1)^{j-1}\tr_{\,\calf}(T_{r-j}(A)\,{\mathcal R}^P_{A^{j-1}Z}) \big)\, {\rm d}\vol_L =0.
\end{eqnarray*}
\end{theorem}

\begin{corollary}
Let $\tau_1={\rm const}$ and $\Ric^P_{N,N}
%\tr_{\,\calf}{\mathcal R}^P_{N}
>0$. Then $\calf$ has no compact leaves.
%on a neighborhood of a compact leaf $L$ and $\tr_{\,\calf}{\mathcal R}^P_{N}\ge0$,
%then $A=0=Z$ and $\tr_{\,\calf}{\mathcal R}^P_{N}=0$ on $L$.
\end{corollary}

\begin{proof}
For $r=0$ and $Z=P\nabla_NN$, by Proposition~\ref{P-Cintfiv01p1},
\begin{equation*}
 \Div_\calf {Z} =\tau_{2}- N(\tau_{1})+\tr_{\,\calf}{\mathcal R}^P_{N}+\<{Z}, {Z}\>.
\end{equation*}
By~conditions, $N(\tau_{1})=0$ and $\tr_{\,\calf}{\mathcal R}^P_{N}=\Ric^P_{N,N}>0$, hence $\Div_\calf {Z}>0$. Applying the Divergence Theorem to a compact leaf yields a contradiction.
%Thus, if $\tau_1={\rm const}$ on a neighborhood of a compact leaf $L$ and $\tr_{\,\calf}{\mathcal R}^P_{N}\ge0$,
%then $A=0=Z$ and $\tr_{\,\calf}{\mathcal R}^P_{N}=0$ on $L$.
\end{proof}

For any vector field $X$ in ${\mathcal D}$, we have
\begin{equation}\label{E-divX}
 \Div X = \Div_\calf X -\<X, \,Z\> -\<X, \,H^\bot\>,
\end{equation}
where $H^\bot$ is the mean curvature vector field of the distribution ${\mathcal D}^\bot$ in $(M,g)$.
Recall that a distribution on a Riemannian manifold is called \textit{harmonic} if its mean curvature vector field vanishes.
There are topological restrictions for the existence of a Riemannian metric on closed manifold,
for which a given distribution becomes harmonic, see \cite{su2}.

The following statement generalizes 
%the result of G.\,Reeb \cite{reeb1}, see 
\eqref{E-sigma1}.

\begin{theorem}
%\label{T-main-sigma1}
For a closed sub-Riemannian manifold $(M,{\mathcal D},g)$ with ${\mathcal D}=T\calf\oplus\,{\rm span}(N)$
and a harmonic orthogonal distribution ${\mathcal D}^\bot$, the following integral formula is valid:
% for $r\le n-2$:
\begin{equation}\label{E-int-sigma1}
 \int_M \sigma_{1}\,{\rm d}\vol_g = 0.
\end{equation}
\end{theorem}

\begin{proof} Recall that $\sigma_{1}=\tr A$, see \eqref{E-A-D}, and observe that
\[
 \Div_\calf N = \sum\nolimits_{\,1\le i\le n}\<\nabla_{e_i}\,N, \,e_i\> = -\sigma_1.
\]
Thus,
\begin{equation}\label{E-divN}
 \Div N = \Div_\calf N -\<N,\,H^\bot\> = -\sigma_1 -\<N,\,H^\bot\>.
\end{equation}
Applying the Divergence Theorem and using the assumption $H^\bot=0$, yields \eqref{E-int-sigma1}.
\end{proof}

The following integral formula generalizes \eqref{E-intNTp1}.

\begin{theorem}\label{T-main01p1}
For a closed sub-Riemannian manifold $(M,g)$ with ${\mathcal D}=T\calf\oplus\,{\rm span}(N)$
and a harmonic orthogonal distribution ${\mathcal D}^\bot$, the following integral formula is valid:
\begin{eqnarray}\label{E-intNTNTp}
\nonumber
 && \int_M \Big( (r+2)\,\sigma_{r+2} -\tr_{\,\calf}(T_r(A){\mathcal R}^P_{N})\\
 &&  -\sum\nolimits_{\,1\le j\le r}(-1)^{j-1}\tr_{\,\calf}(T_{r-j}(A){\mathcal R}^P_{A^{j-1}\nabla^P_NN}) \Big)\,{\rm d}\vol_g = 0.
\end{eqnarray}
\end{theorem}

\begin{proof}
Using \eqref{E-divN}, we calculate the divergence of $\sigma_{r+1}\cdot N$ as
\[
 \Div(\sigma_{r+1}\cdot N) =  \sigma_{r+1}\Div N + N(\sigma_{r+1}) = N(\sigma_{r+1}) - \sigma_{1}\sigma_{r+1}.
\]
Then, by \eqref{E-divX} with $X=T_r(A)Z$,
(since ${\mathcal D}^\bot$ is a harmonic distribution), we get
\begin{equation*}
 \Div\big(T_r(A) Z  +\sigma_{r+1}\cdot N\big) = \<\Div_\calf T_r(A), \,Z\> -(r+2)\,\sigma_{r+2} +\tr_{\,\calf}(T_r(A) {\mathcal R}^P_{N}).
\end{equation*}
Thus, by Proposition~\ref{P-Cintfiv01p1}, we get \eqref{E-intNTNTp}.
\end{proof}

\begin{corollary}
%\label{C-sigma2}
For $r=0$, \eqref{E-intNTNTp} gives us the following generalization of \eqref{E-sigma2}:
\begin{equation}\label{E-sigma2-new}
 \int_M (2\,\sigma_2-\Ric^P_{N,N})\,{\rm d}\vol_g=0.
\end{equation}
For~$n=1$, we get $\sigma_2=0$, thus, \eqref{E-sigma2-new} gives the zero integral of ``Gaussian $P$-curvature" of~${\mathcal D}$.
%%\[
% $\int_M K^P\,{\rm d}\vol_g=0$,
%%\]
%where $K^P=\<R^P(X,N,N,X)$ and $X\in{\mathcal D}$ is a unit vector of $\,T\calf$.
\end{corollary}

\begin{example}\rm
(a) For $r=1,2$, (\ref{E-intNTNTp}) reduces to
\begin{equation*}
%\label{E-inttauk2}
 \int_M\Big(3\,\sigma_3 -\tr_{\,\calf}(T_1(A)\,{\mathcal R}^P_{N} +{\mathcal R}^P_{Z})\Big)\, {\rm d}\vol_g=0,
\end{equation*}
% For $r=2$, (\ref{E-intNTNTp}) gives
% a  similar to (\ref{E-intCh1}) integral formula
\begin{equation}\label{E-intNTNT-k2}
 \int_{{\mathcal N}_1\calf}\Big(4\,\sigma_{4}+\<T_2(A) Z, {H}^\bot\>
 -\tr_{\,\calf}\big(T_2(A){\mathcal R}^P_{\,N} + T_1(A){\mathcal R}^P_{\,Z} -{\mathcal R}^P_{\,A Z}\big)\Big)\,{\rm d}\,\omega^\perp = 0.
\end{equation}
(b) If $N$ is a $P$-\textit{geodesic vector field}, i.e., $Z=0$, then (\ref{E-intNTNTp}) shortens to the formula
\begin{equation*}
%\label{E-intNTtotG2}
 \int_{{\mathcal N}_1\calf}\big( (r+2)\,\sigma_{r+2} -\tr_{\,\calf}(T_r(A){\mathcal R}^P_{N}) \big)\,{\rm d}\,\omega^\bot=0.
\end{equation*}
\end{example}

\section{Some consequences}

Here, we apply our formulas to sub-Riemannian manifolds with restrictions on the shape operator $A$ or the induced curvature tensor $R^P$.

Let $(M,{\mathcal D},\calf,g)$ be a foliated sub-Riemannian manifold with ${\mathcal D}=T\calf\oplus\,{\rm span}(N)$, then
% will be called

%(i) $N$ will be called $P$-\textit{geodesic}, if $P\nabla_NN = 0$,
%$(\nabla_{e_a}\,e_b)^\top=0$ for all $a,b$,

$\bullet$~$\calf$ will be called $P$-\textit{harmonic}, if
%$Z=0$ (or,
$\sigma_1=0$,

$\bullet$~$\calf$ will be called $P$-\textit{totally umbilical}, if
%\begin{equation}\label{E-P-umb}
 $A=(\sigma_1/n)\id_{\,T\calf}$.
%\end{equation}

\smallskip\noindent
Obviously,
%$P$-{geodesic},
$P$-{harmonic} and $P$-{totally umbilical} distributions are
%{auto-paral\-lel},
{harmonic} and {totally umbilical}, respectively, but the opposite is not true.
By \eqref{E-sigma2-new}, we obtain the following.

\begin{corollary}
Let $(M,{\mathcal D},g)$ be a closed sub-Riemannian manifold with
%${\mathcal D}=T\calf\oplus\,{\rm span}(N)$ and
a harmonic orthogonal distribution ${\mathcal D}^\bot$.

\noindent\
{\rm(i)}~If $\,\Ric^P>0$, then there are no $P$-harmonic codimen\-sion-one foliations in ${\mathcal D}$.
% with
%$\tr_{\,\calf}{\mathcal R}^P_{N}
%$\Ric^P_{N,N}>0$.

\noindent\
{\rm(ii)}~If $\,\Ric^P<0$, then there are no $P$-totally umbilical codimension-one foliations in ${\mathcal D}$.
%of a closed manifold of negative
%with
%$\tr_{\,\calf}{\mathcal R}^P_{N}
%$\Ric^P_{N,N}<0$,
\end{corollary}

\begin{proof}
(i)~If $\calf$ is a $P$-harmonic codimension-one foliation in ${\mathcal D}$, then $2\,\sigma_2=-\tau_2\le 0$, see \eqref{E-sigma2-tau1-2}.
(ii)~If $\calf$ is a $P$-totally umbilical codimension-one foliation in ${\mathcal D}$, then $2\,\sigma_2=\frac{n-1}n\,(\sigma_1)^2\ge 0$.

In both cases, (i) and (ii), we get a contradiction to integral formula \eqref{E-sigma2-new}.
\end{proof}

The \textit{total $r$-th mean curvature\/} of $\calf$ are defined by
\[
 \sigma_{r}(\calf)=\int_{M}\sigma_{r}\,{\rm d}\vol_g.
\]
The~following corollary of Theorem~\ref{T-main01p1} generalizes \cite[Theorem~1.1]{blr}
(see also \cite[Section~4.1]{lw2}).

\begin{corollary}\label{C-P1}
Let $(M,{\mathcal D},g)$ be a closed sub-Riemannian manifold with ${\mathcal D}=T\calf\oplus\,{\rm span}(N)$
and a harmonic orthogonal distribution ${\mathcal D}^\bot$, satisfying condition \eqref{E-Pcurv-c}.
Then $\sigma_{r}(\calf)$ depends on $r,n,c$ and the volume of $(M,g)$ only, i.e., the following integral formula is valid:
%Let $(M,{\mathcal D},g)$ be a closed sub-Riemannian manifold with a harmonic distribution ${\mathcal D}^\bot$
%and condition \eqref{E-Pcurv-c}. Then the following integral formula is valid:
\begin{equation*}
%\label{E-RP-c}
  \sigma_{r}(\calf) =
  \bigg\{\begin{array}{cc}
   c^{r/2}\Big(\begin{smallmatrix} n/2 \\ r/2 \end{smallmatrix}\Big) {\rm Vol}(M,g), & n,r\ {\rm even}, \\
   0 , & n\ {\rm or}\ r\ {\rm odd}.
  \end{array}%\right.
\end{equation*}
\end{corollary}

\begin{proof} Using Lemma~\ref{L-NTprop}, we get
\[
 \tr_{\,\calf}(T_r(A){\mathcal R}^P_N) = \sum\nolimits_{\,1\le i\le n} \<R^P(e_i,N)N, \,T_r(A)e_i\>
 = c\,\tr_{\,\calf}T_r(A) = c(n-r)\,\sigma_r.
\]
From \eqref{E-intNTNTp} we obtain the equality
\begin{equation}\label{E-cor1}
 (r+2)\,\sigma_{r+2}(\calf) =  c(n-r)\,\sigma_{r}(\calf) .
\end{equation}
Since $\sigma_{1}(\calf) =0$, see \eqref{E-int-sigma1}, by induction we obtain $\sigma_{r}(\calf) =0$ for any odd $r$.
For even $r=2s$ and $n=2l$, using \eqref{E-cor1} and induvtion implies
$\sigma_{2s}(\calf) =c^{s}\big(\begin{smallmatrix} l \\ s \end{smallmatrix}\big) {\rm Vol}(M,g)$.
\end{proof}

We can get similar formulas when $(M,{\mathcal D},g)$ has the property
\begin{equation}\label{E-P-Einst}
% \tr_{\,\calf}{\mathcal R}^P_{X}
 \Ric^P_{X,N}=C\,\<X,\,N\>,\quad X\in{\mathcal D},
\end{equation}
for some $C\in\RR$ and $\calf$ (with $\dim\calf>1$) is $P$-totally umbilical.
Note that for ${\mathcal D}=TM$, condition \eqref{E-P-Einst} is satisfied for Einstein manifolds, see \cite{lw2}.

The~following corollary of Theorem~\ref{T-main01p1} generalizes result in \cite[Section~4.2]{lw2}.

\begin{corollary}\label{C-P2}
Let $(M,{\mathcal D},g)$ be a closed sub-Riemannian manifold with
%a distribution
${\mathcal D}=T\calf\oplus\,{\rm span}(N)$, $P$-\textit{totally umbilical} foliation $\calf$, a harmonic orthogonal distribution ${\mathcal D}^\bot$,
and satisfying \eqref{E-P-Einst}.
Then $\sigma_{r}(\calf)$ depends on $r,n,C$ and the volume of $(M,g)$ only, i.e., the following integral formula holds:
\begin{equation}\label{E-cor2}
 %\int_M \sigma_{r}\,{\rm d}\vol_g
  \sigma_{r}(\calf) =
  \bigg\{\begin{array}{cc}
   (C/n)^{n/2}\Big(\begin{smallmatrix} n/2 \\ r/2 \end{smallmatrix}\Big) {\rm Vol}(M,g), & n,r\ {\rm even}, \\
   0 , & r\ {\rm odd}.
  \end{array}%\right.
\end{equation}
\end{corollary}

\begin{proof}
In this case, $T_r(A)$ has the form
%, see also \cite[Section~4.2]{lw2},
\[
 T_r(A) = a_r\,(\sigma_1/n)^r\id_{\,T\calf},\quad {\rm where}\quad a_r=\sum\nolimits_{\,0\le i\le r}(-1)^{r-i}
 \Big(\begin{matrix} n \\ i \end{matrix}\Big) .
\]
Since $\tr_{\,\calf} T_r(A)=(n-r)\sigma_r = (n-r)\big(\begin{smallmatrix} n \\ r \end{smallmatrix}\big)(\sigma_1/n)^r$, then
\[
 a_r = \frac{n-r}{n}\Big(\,\begin{matrix} n \\ r \end{matrix}\,\Big),\quad
 T_r(A) = \frac{n-r}{n}\,\sigma_r\id_{\,T\calf}.
\]
By our assumption \eqref{E-P-Einst},
\[
 %\tr_{\,\calf}{\mathcal R}^P_Z
 \Ric^P_{Z,N}=0,\quad
 %\tr_{\,\calf}{\mathcal R}^P_N
 \Ric^P_{N,N} = C.
\]
Thus, for a closed manifold with a $P$-\textit{totally umbilical} foliation $\calf$ and condition \eqref{E-P-Einst}, formula \eqref{E-intNTNTp} becomes
\begin{equation*}
%\label{E-intNTNTp}
 \sigma_{r+2}(\calf) = \frac{C(n-r)}{n(r+2)}\,\sigma_{r}(\calf).
\end{equation*}
Using induction similarly to Corollary~\ref{C-P1}, we obtain \eqref{E-cor2}.
\end{proof}

\begin{remark}\rm
Our integral formulas provide more conditions for the mean curvature $H=\sigma_1/n$ of $\calf$.
In the case of a $P$-totally umbilical codimension-one foliation in ${\mathcal D}$ with $A=H\id_{\,T\calf}$, such conditions can be easily derived from \eqref{E-intNTNTp} using $\sigma_r = \big(\begin{smallmatrix} n \\ r \end{smallmatrix}\big) H^r$:
\begin{eqnarray}\label{E-intNTNTp-H}
\nonumber
 && \int_M H^{r-1}\Big({H}^{3}(n-1)(n-r-1)\,n \\
 && -{H}(n-1)(r+1)\Ric^P_{\,N,N} - r(r+1)\Ric^P_{\,Z,N}\Big)\,{\rm d}\vol_g = 0.
\end{eqnarray}
%${n\choose r}\frac{n-r}{(r+1)\,r(n-1)\,n}\left({H}^{3}(n-1)(n-r-1)nr -{H}(n-1)n(r+1)\Ric^P_{\,N,N} -{r}^{2}(r+1)\Ric^P_{\,Z,N}\right)$
Here we used the following identity with binomial coefficients:
\[
 \sum\nolimits_{\,1\le j\le r}\,(-1)^{j-1}\,\frac{n-r+j}{n} \Big(\,\begin{matrix} n \\ r-j \end{matrix}\Big)
 =\Big(\,\begin{matrix} n -2 \\ r-1 \end{matrix}\Big).
\]
These integrals contain polynomials depending on $H$, and one can get obstructions for existence of $P$-totally umbilical foliations.
For example, if $r=n-1$, then \eqref{E-intNTNTp-H} reads as
$\int_M H^{n-2}(H\Ric^P_{\,N,N}+\Ric^P_{\,Z,N}) \,{\rm d}\vol_g=0$.
\end{remark}

\begin{example}\rm
Here is an amazing consequence of \eqref{E-int-sigma1} and \eqref{E-sigma2-new}.
Let a compact sub-Riemannian mani\-fold $(M, g; {\mathcal D})$ with a harmonic orthogonal distribution ${\mathcal D}^\bot$ and the condition
\[
 \Ric^P\ge 2\,c >0
 %,\quad \mbox{for some real }\ c,
\]
for some real $c$, be equipped with a codimension-one foliation $\calf$, i.e., ${\mathcal D}=T\calf\oplus\,{\rm span}(N)$.
Then \textit{the image of the function $\sigma_2: M \to \RR$ contains interval $[0, c+\eps]$ for some $\eps>0$}.
Indeed, by Reeb type formula \eqref{E-int-sigma1}, $\sigma_1(x)=0$ at some $x\in M$. Then $\sigma_2(x)=-\tau_2(x)\le 0$, see \eqref{E-sigma2-tau1-2}.
By \eqref{E-sigma2-new}, $\sigma_2(\calf)\ge c\,{\rm Vol}(M,g)$. Hence, there exists $y\in M$ such that $\sigma_2(y) > c$
(otherwise, $\sigma_2 \le c$ on $M$, therefore, $\sigma_2 \equiv c>0$ on $M$ -- a contradiction to $\sigma_2(x) \le 0$).
\end{example}

\section{Conclusion}

In the article we generalized integral formula \eqref{E-intNTp1} (and its consequences) for sub-Riemannian setting.
%We suggest that integral formulas \eqref{E-intCh-B}, \eqref{E-intCh-C}, \eqref{E-intCh}, \eqref{E-intNTNT} are a very good tool for understanding the geometry %of (sub-)Riemannian manifolds equipped with foliations or distributions.
We delegate the following for further study.

1. Our integral formulas can be easily extended for foliations and distributions defined outside of~a~``singularity~set" $\Sigma$
(a finite union of pairwise disjoint closed submanifolds of codimension at least $k$ of a closed manifold $M$)
under additional assumption of convergence of certain integrals.
Then, instead of the Divergence theorem, we apply the following, see~\cite{lw2,pw1}:
if $(k-1)(q-1)\ge1$ and $X$ is a vector field on $M\setminus\Sigma$ such that $\int_M\|X\|^q\,{\rm d}\vol_g < \infty$, then
%\begin{equation}
%\label{E-L2}
$\int_M \Div X\,{\rm d}\vol_g = 0$.
%\end{equation}

2. One can extend our integral formulas for holomorphic foliations of
%K\"{a}hler manifolds.
complex sub-Rieman\-ni\-an manifolds, see \cite{sve} for the case of Riemannian manifolds, i.e., ${\mathcal D}=TM$.

3. Our integral formulas can be extended to foliations of arbitrary codimension of a sub-Riemannian manifold,
see \cite{r8} for the case of Riemannian manifolds, i.e., ${\mathcal D}=TM$.

%\baselineskip=12.pt

% ------------------------------------------------------------------------
\end{document}